\documentclass[11pt, reqno]{amsart}
\usepackage{amsmath,amsfonts,amssymb, tikz}
\usepackage[utf8]{inputenc}
\usepackage[english]{babel}
\usepackage{xcolor}
\usepackage{enumerate}
\usepackage[tableposition=top]{caption} 
\newtheorem{theorem}{Theorem}[section]
\newtheorem*{theorem*}{Theorem}
\newtheorem{lemma}{Lemma}[section]

\newtheorem{proposition}{Proposition}[section]

\theoremstyle{remark}

\newtheorem{remark}{Remark}

\numberwithin{equation}{section}

\newcommand{\aut}{\mbox{Aut}}

\title{On Moments of non-normal number fields}
\author{Kalyan Chakraborty}
\email[Kalyan Chakraborty]{kalychak@ksom.res.in}
\address{ Kerala School of Mathematics, KCSTE,
	Kunnamangalam, Kozhikode, Kerala, 673571, India.}

\author{Krishnarjun K}
\email[Krishnarjun K]{krishnarjunk@hri.res.in, krishnarjunmaths@gmail.com}
\address{Harish-Chandra Research 
	Institute, HBNI,
	Chhatnag Road, Jhunsi, Prayagraj 211019, India.}

\keywords{Artin $ L $ function, Non-normal fields, Moments, Dedekind zeta function, Finite group representations.}
\subjclass[2020] {11F66, 11F30, 11R42, 20C30}

\begin{document}
	
	\maketitle
	
	\begin{abstract}
		Let $ K $ be a number field over $ \mathbb{Q} $ and let $ a_K(m) $ denote the number of integral ideals of $ K $ of norm equal to $ m\in\mathbb{N} $. In this paper we obtain asymptotic formulae for sums of the form $ \sum_{m\leq X} a^l_K(m) $ thereby generalizing the previous works on the problem. Previously such asymptotics were known only in the case when $ K $ is Galois or when $K$ was a non normal cubic extension and $ l=2,3 $. The present work subsumes both these cases.
	\end{abstract}
	
	\section{Introduction}\label{Section "Introduction"}
	
	Let $ K $ be a finite extension of $ \mathbb{Q} $ and let $ a_K(m) $ denote the number of integral ideals of $ K $ of norm equal to $ m\in \mathbb{N} $. For any other natural number $ l $, the $ l $-th moment of the Dedekind zeta function associated to $ K $ is defined as the sum
	
	\begin{equation}\label{Equation "Definition Moments"}
		\mathcal{S}_l(X) = \sum_{m\leq X} a_K^l(m).
	\end{equation}
	
	As a consequence of the approximate functional equation developed by Chandrasekharan and Narasimhan \cite{Chand Nar}, Chandrasekharan and Good \cite{Chand Good} gave the following asymptotic formulae for the moments of the Dedekind zeta function of Galois extensions of $ \mathbb{Q} $.
	\begin{theorem}[Chandrasekharan - Good]\label{Theorem "Chand Good"}
		If $ K $ is a Galois extension of $ \mathbb{Q} $ of degree $ n > 1 $, then for every $ \epsilon >0 $ and any integer $ l\geq 2 $, we have
		\[
		\sum_{m\leq X} a_K(m)^l = XP_l(\log(X)) + O(X^{1-2n^{-l}+\epsilon})
		\]
		where $ P_l $ denotes a suitable polynomial of degree $ n^{l-1}-1 $.
	\end{theorem}
	
	Earlier, Chandrashekaran and Narasimhan studied the special case of the above result for $ l=2 $ for all number fields. In the Galois case, they were able to establish the asymptotic similar to Theorem \ref{Theorem "Chand Good"} above but for the non Galois case they were only able to establish an upper bound for the moments.
	
	\begin{theorem}[Chandrasekharan - Narasimhan]\label{Theorem "Chand Nar"}
		If $ K $ is a number field of degree $ n $ over $ \mathbb{Q} $, and $ a_K(m) $ as above, then 
		\[
		\sum_{m\leq X} a_K^2(m) \ll X(\log(X))^{n-1}.
		\]
	\end{theorem}
	
	In general the asymptotics of sums of the type $ \mathcal{S}_l(X) $ are of the form $ XP_l(\log(X)) $ for a polynomial $ P_l $ of some suitable degree. This is a direct consequence of the existence of a high order pole of certain $ L $ series at the abscissa of convergence. Interestingly, we shall see that the upper bound in Theorem \ref{Theorem "Chand Nar"} is rarely sharp and the actual sum grows slower in many orders of magnitude.
	
	Using the strong Artin conjecture, Fomenko \cite{Fomenko} proved the asymptotic relations for the case of a non normal cubic field whose Galois closure has Galois group isomorphic to $ S_3 $ and estimated the first and second moments. More precisely he proved the following theorem.
	
	\begin{theorem}[Fomenko]\label{Theorem "Fomenko"}
		Let $ K $ be a non normal extension of $ \mathbb{Q} $ of degree $ 3 $ such that the Galois closure of $ K $ over $ \mathbb{Q} $ has Galois group isomorphic to $ S_3 $. Then
		\begin{align*}
			\sum_{m\leq X} a^2_K(m) &= C_1X\log(X) + C_2X + O(X^{9/11+\epsilon})\\
			\sum_{m\leq X} a^3_K(m) & = XP_3(\log(X)) + O(X^{73/79 + \epsilon}),
		\end{align*}
		where $ P_3 $ is a polynomial of degree $ 4 $.
	\end{theorem}
	
	The exponents in the error terms were improved by L\"u \cite{Lu1}.
	
	\begin{theorem}[L\"u]\label{Theorem "Lu"}
		Let $ K $ be as in Theorem \ref{Theorem "Fomenko"}, then 
		\begin{align*}
			\sum_{m\leq X} a^2_K(m) &= C_1X\log(X) + C_2X + O(X^{23/31+\epsilon})\\
			\sum_{m\leq X} a^3_K(m) & = XP_3(\log(X)) + O(X^{235/259 + \epsilon}).
		\end{align*}
	\end{theorem}
	
	However their method was limited to the first and second moments as it relied heavily on deep results from the Langlands programme such as the truth of the strong Artin conjecture and the automorphy of symmetric powers of certain $ GL(2) $ automorphic representations. In this paper we use different methods and study the problem for a general class of non normal extensions whose Galois closure has Galois group which can be written as a semi direct product with certain mild conditions. The novelty of the present work is two fold. We prove the asymptotic estimate for the moments for a general class of number fields whose Galois closures can have a wide range of Galois groups (we have provided an incomplete list of the possible choices in \S \ref{Section "Some Examples"}) and in many cases we are able to successfully estimate arbitrary high moments as well.
	
	Since the Galois case was already considered by Chandrasekharan and Good, the focus of this paper is on the non Galois extensions of $ \mathbb{Q} $. Nevertheless, our methods apply for Galois extensions as a trivial case. Therefore the statements and proofs that follow are primarily concerned with the non Galois case, of which the Galois case becomes a trivial corollary (see \S\ref{Subsection "The Galois Case"}).
	
	Now we shall state the main results of this paper. Let $K$ be a non normal extension of $ \mathbb{Q} $ and let $ K' $ be its Galois closure. We denote $ Gal(K'/\mathbb{Q}) $ as $ G $. Suppose that $ G $ can be written as a semi direct product of two non trivial groups $ N,H $, that is $ G=N\rtimes_\varphi H $ where $ \varphi :H\to Aut(N) $ is a non trivial homomorphism (for details see \S \ref{Section "A slightly general result"}). Suppose that $ Gal(K'/K) = N'\rtimes_\varphi H$ for some normal subgroup $ N' $ of $ N $.
	
	\begin{theorem}\label{Theorem "Main Theorem"}
		Suppose $ K $ is as above. Furthermore, suppose that conditions (\ref{Assumption 1}), (\ref{Assumption 2}) and (\ref{Assumption 3}) (see \S \ref{Section "Analytic continuation"}) are satisfied. Then, for every $ \epsilon > 0 $, we have
		\[
		\mathcal{S}_l(X) = XP_l(\log(X)) + O_\epsilon (X^{\delta+\epsilon}),
		\]
		where $ \delta < 1 $ and $ P_l $ is a polynomial whose degree can be explicitly calculated depending on $ G $ and $ K $.
	\end{theorem}
	
	In fact, we give explicit formula for $ \delta $ and the degree of $ P_l $ in terms of $ G $. The theorem is proved in two steps. In order to estimate the sum $ \mathcal{S}_l(X) $ we study the $ L $ series defined as
	\begin{equation}\label{Equation "Moment L series definition"}
		D_l(s) : = \sum_{m=1}^\infty \frac{a_K^l(m)}{m^s}.
	\end{equation}
	
	\subsection*{Notations and conventions}
	The symbols $ K,K',N,H $ will carry the same meaning throughout the paper as they carried in \S \ref{Section "Introduction"}. 
	All representations considered will be complex representations. If $ G $ is a group, the identity element of $ G $ will be denoted as $ e_G $.
	
	The paper is structured as follows. In \S \ref{Section "Preliminaries"} we recall some basic facts about certain $ L $ series associated to number fields, in particular the Artin $ L $ series. In \S \ref{Section "A slightly general result"} we prove a slightly general result (see Proposition \ref{Proposition "The general result"}) regarding some divisibility properties of $ a_K(p) $ with a rather minimal amount of conditions on $ K $. This can be seen as an analogue of Lemma 1 of \cite{Chand Good} for the non Galois case. In \S \ref{Section "Analytic continuation"} we study an $ L $ series associated to $ \mathcal{S}_l(X) $ and under some additional assumptions on the structure of $ G $ we show that the $ L $ series can be analytically continued (see Theorem \ref{Theorem "Analytic continuation"}) to the region $ Re(s)>1/2 $. In \S \ref{Section "Estimates on Moments"} we give asymptotic formulae for the moments. Finally in \S \ref{Section "Some Examples"} we give an incomplete
	list of the various examples which fall under the present work so as to illustrate that the restrictions imposed on $ G $ are not too strong.
	
	\section{Preliminaries}\label{Section "Preliminaries"}
	
	In this section we shall recall some basic facts about some of the objects we shall be working with and fix the notations along the way.
	
	\subsection{Dedekind zeta function and Hecke $ L $ functions}
	Let $K$ be a number field of degree $n$ over $\mathbb{Q}$. Then the Dedekind zeta function, $\zeta_K$ of $K$ is defined as
	\[
	\zeta_K (s) = \sum_{\mathfrak{a}\neq 0} \frac{1}{(N\mathfrak{a})^s},
	\]
	where the summation runs over all non-zero integral ideals of $K$ and $N(\mathfrak{a})$ denotes the ideal norm of $\mathfrak{a}$. 
	For $ Re(s) >1 $, the Dedekind zeta function can be rewritten as 
	\begin{equation}\label{Equation "Dedekind Zeta function"}
		\zeta_K(s) = \sum_{m=1}^{\infty} \frac{a(m)}{m^s}.
	\end{equation}
	
	It is known that $ a(m)\ll_\epsilon m^\epsilon $ (see \cite{Chand Nar} for example). Therefore $\zeta_K$ converges absolutely for $Re(s)>1$ and admits a meromorphic continuation to the entire complex plane with a simple pole at $s=1$. The residue at this pole is given by the class number formula which states that
	\[
	\underset{s=1}{\mbox{Res}}\ \zeta_K(s) = \frac{2^{r_1}(2\pi)^{r_2}hR}{w\sqrt{|\Delta_K|}},
	\]
	where $h, R, w$ are respectively the class number, regulator and the number of roots of unity of $K$. 
	
	Given a number field $ K $, a $ K $ modulus is a formal product of an integral ideal of $ K $ along with a (possibly empty) collection of real places of $ K $. Given a $ K $ modulus $ \mathfrak{M} $ it is possible to associate a ray class group which is a finite quotient of the group of fractional ideals coprime to $ \mathfrak{M} $ of $ K $. Characters of this ray class groups which satisfy an additional condition are called the Hecke characters associated to $ K $. Hecke characters are natural generalizations of Dirichlet characters, and the related theory is well understood.
	
	Given a Hecke character $ \Psi $ it is possible to associate an $ L $ series to $ \Psi $ as follows
	\begin{equation}\label{Equation "Hecke L function definition"}
		L(s, \Psi ) = \sum_{\mathfrak{a}} \frac{\Psi(\mathfrak{a})}{N(\mathfrak{a})^s},
	\end{equation}
	where the summation runs over non zero integral ideals of $ K $. The Hecke $ L $ functions are well studied and many properties such as analytic continuation and functional equation are well understood. In particular, if the Hecke character is the trivial character, we get back the Dedekind zeta function.
	
	\subsection{Artin $ L $ functions}
	
	Let $ K $ be a Galois extension of $ \mathbb{Q} $ with Galois group $ G $. Let $ \rho $ be a representation of $ G $. For every prime $ \mathfrak{p} $ of $ K $, there is an element $ \sigma_\mathfrak{p}\in G $ called the Frobenius at $ \mathfrak{p} $. If $ \mathfrak{p}_1 $ and $ \mathfrak{p}_2 $ are two primes in $ K $ lying above an unramified prime $ p $, then the Frobenius elements $ \sigma_{\mathfrak{p}_1} $ and $ \sigma_{\mathfrak{p}_2} $ are conjugate.
	
	Artin $ L $ functions are defined as Euler products. Given an $ n $ dimensional representation $ \rho $, at every \textit{unramified} prime $ p $, the local factor is defined as 
	
	\[
	\det(I - p^{-s} \rho(\sigma_\mathfrak{p}))^{-1},
	\]
	where $ I $ is the $ n\times n $ identity matrix and $ \mathfrak{p} $ is a prime over $ p $. The Artin $ L $ function is defined as the following product over the unramified primes
	
	\begin{equation}\label{Equation "Artin L function definition"}
		L(s,\psi)	: = \prod_p \det(I - p^{-s} \rho(\sigma_\mathfrak{p}))^{-1}.
	\end{equation}
	We remark here that the local factors at ramified primes can also be defined, but as they require some amount of terminology and notation, we refrain from describing them. In any case, the \textit{analytic properties} that we shall be interested in are not affected in any serious way by the exclusion of the ramified factors as there are only finitely many of them.
	
	Artin $ L $ functions behave well with direct sums of representations. If $ \rho_1 $ and $ \rho_2 $ are two representations, then $ L(s,\rho_1\oplus\rho_2) = L(s,\rho_1)L(s,\rho_2) $. Furthermore, Artin $ L $ functions are invariant under induction. More precisely, if $ \rho_1 $ is a representation of $ H\leq G $ and if $ \rho_2 = \mbox{Ind}_H^G \rho_1 $, then $ L(s,\rho_1) = L(s,\rho_2) $.
	
	Artin reciprocity is one of the crowning achievements of class field theory in the early twentieth century. In terms of $ L $ functions, Artin reciprocity forms a connection between Artin $ L $ functions of $ 1 $ dimensional representations (characters) of the Galois group and $ L $ functions associated to Hecke characters. In other words, it states for every character $ \chi $ of the Galois group, there is a Hecke character $ \Psi $ such that $ L(s, \chi) = L(s,\Psi) $.
	
	The Dedekind zeta function can be written as a product of Artin $ L $ functions. Let $ K $ be a number field, and let $ K' $  be the Galois closure of $ K $ with Galois group $ G $ over $ \mathbb{Q} $. Let $ H\leq G $ be the subgroup of $ G $ associated to $ K $. Consider the trivial representation of $ H $, say we denote by $ \textbf{1}_H $ and let $ \rho = \mbox{Ind}_H^G \textbf{1}_H $ be the induced representation. Then the Dedekind zeta function $ \zeta_K(s) = L(s,\rho) $. Artin holomorphy conjecture states that the Artin $ L $ functions associated to non-trivial irreducible representations are entire functions.
	
	For more details on the Artin $ L $ functions and the Artin conjecture, we refer the reader to the excellent articles \cite{KnappIntro} and \cite{Prasad}.

	\section{A slightly general result}\label{Section "A slightly general result"}		 
	
	Given two groups $ H,N $ and a group homomorphism $ \varphi : H\to \aut(N) $ we define the semi direct product of $ H $ and $ N $ along $ \varphi $ as the set $ N\times H $ with the group operation $ (n_1,h_1)\cdot(n_2,h_2) = (n_1\varphi(h_1)(n_2), h_1h_2) = (n_1\varphi_{h_1}(n_2),h_1h_2))$. We denote this group by $G:= N\rtimes_\varphi H $. There exists subgroups of $ G $ isomorphic to $ N,H $ which we again denote by $ N,H $ respectively. Furthermore we have $ N\cap H = e_G $, the identity element and $ N $ is normal in $ G $. We shall denote each element in $ G $ as $ nh $ rather than $ (n,h) $ whenever $ n\in N $ and $ h\in H $. A few handy identities are $ hn = \varphi_h(n)h $, $ \varphi_h(\varphi_{g}(n)) = \varphi_{hg}(n) $, $ \varphi_h(m)\varphi_h(n) = \varphi_h(mn) $ for all $ m,n\in N $ and $ h,g\in H $. We also define $ \varphi_{H'}(N')$ as the smallest subgroup of $ G $ containing $ \{\varphi_h(n)\ |\ n\in N', h\in H'\} $ for subgroups $ H' $ of $ H $ and $ N' $ o
	f $ N $. We observe that $ \varphi_{H'}(N') $ is stable under the action of $ H' $ and therefore we can define $ \varphi_{H'}(N')\rtimes_\varphi H'\hookrightarrow G $.
	
	\begin{lemma}\label{Lemma "Equivalence of normality"}
		Let $ N' $ be a subgroup of $ N $ and let $ G=N\rtimes_\varphi H $. Then the following are equivalent,
		\begin{enumerate}
			\item	$ N' $ is normal in $ N $ and $ \varphi_{H}(N')\subseteq N' $,
			\item	$ N' $ is normal in $ G $.
		\end{enumerate}
	\end{lemma}
	
	\begin{proof}
		Suppose $ m\in N' $ and $ nh\in G $. Then 
		\begin{equation}\label{Equation "Lemma 1"}
			nhmh^{-1}n^{-1} = n\varphi_h(m)n^{-1}.
		\end{equation}
		Suppose $ (1) $ is true, then clearly $ N' $ is normal in $ G $. Conversely, if $ (2) $ is true, then by taking $ h=e_H $ in \eqref{Equation "Lemma 1"} we see that $ N' $ is normal in $ N $. Then $ n\varphi_h(m)n^{-1}\in N' $ if and only if $ \varphi_h(m)\in N' $. Therefore $ \varphi_H(N')\subseteq N' $.
	\end{proof}
	
	Let $N'$ be a normal subgroup of $N$ such that $ N' =  \varphi_{H}(N')$. Suppose $ G' := N'\rtimes_\varphi H $, a subgroup of $ G $. We define $ N'' = N/N'$. We fix a set of coset representatives for $N'$ in $N$ and identify them with the elements of $ N''$. In other words, every $n\in N$ can be uniquely written as a product $n_1n_2$ where $n_1\in N''$ and $n_2\in N'$. Also, a set of coset representatives representatives of $ G' $ in $ G $ is given by elements of $ N'' $.
	
	Let $ \textbf{1}_{G'} $ denote the trivial representation of $ G' $ and let $ \rho_1 $ denote the induction of $ \textbf{1}_{G'} $ to $ G $. Let $ \chi_1 $ denote the character of $ \rho_1 $, then for any $ g=n_1n_2h \in G $, where $n_1\in N'', n_2 \in N', h\in H$, we have
	
	\begin{align*}
		\chi_1(g) &= \sum_{n\in N''} \textbf{1}_{G'} (n^{-1}gn)\\
		&= \sum_{n\in N''} \textbf{1}_{G'}(n^{-1} n_1n_2h n)\\
		&=  \sum_{n\in N''} \textbf{1}_{G'}(n^{-1}n_1n_2\varphi_h(n) h).
	\end{align*}
	
	If the right hand side is non zero, then we have $ n^{-1}n_1n_2\varphi_h(n)\in N' $. Since $N'$ is normal in $N$, $n_2\varphi_h(n) = \varphi_h(n)\Tilde{n_2}$ for some $\Tilde{n_2}\in N'$. Therefore we have the relation $n^{-1}n_1\varphi_h(n) \in N'$. Since $H$ fixes $N'$, $\varphi_h$ factors through to a map on $N''$. In this sense, the above relation is equivalent to the condition $n^{-1}n_1\varphi_h(n) = e_{N''}$ or equivalently, $n_1 = n \varphi_h(n)^{-1}$. Suppose $N_h''$ denote the subgroup of $N''$ of fixed points of $\varphi_h$. Then $n\varphi(n)^{-1} = m\varphi_h(m)^{-1}$ implies $nm^{-1} \in \varphi_h$. In other words, the number of solutions for the equation $n_1 = n \varphi_h(n)^{-1}$ is either equal to zero or equal to $|N_h''|$. Conversely, the image of the map (which is not necessarily a group homomorphism) $\phi_h(n) : N''\to N''$ defined as $\phi_h(n) = \varphi_h(n)n^{-1}$ is constant on each coset of $N''_H$ in $N''$.
	
	Summarizing the above discussion we can write
	\begin{equation}\label{Equation "chi_1"}
		\chi_1(n_1n_2h) = \begin{cases}
			|N''_h| & \mbox{if } h\in H\mbox{ and } n_1 = n\varphi_h(n^{-1})\ \mbox{for some } n\in N'',\\
			0 & \mbox{otherwise.}
		\end{cases}
	\end{equation}
	
	\begin{proposition}\label{Proposition "The general result"}
		Let $K$ be a non-normal extension of $\mathbb{Q}$ of degree $d$ whose Galois closure we denote by $K'$. Suppose that $Gal(K'/\mathbb{Q}) = N\rtimes_\varphi H$ and $Gal(K'/K) = N'\rtimes_\varphi H$ as above. Then either $a_K(p)=0$ or $a_K(p)$ divides $d$ for all primes $p$ which are unramified in $ K $.
	\end{proposition}
	
	\begin{proof}
		The proof follows by observing that for a prime $ p $ unramified in $ K $, $a(p) = \chi_1(\sigma_p)$ where $\sigma_p$ is the Frobenius element (upto conjugation) at $p$, $d=|N''|$ and \eqref{Equation "chi_1"}.
	\end{proof}
	
	\begin{remark}\label{Remark 1}
		The map $ \phi_h $ is a $ |N_h''| $ to $ 1 $ map on $N''$. In particular if $ |N_h''| = 1 $ for some $ h $, then $ \phi_h(n) $ mentioned above is a bijection on $ N'' $.
	\end{remark}
	
	\begin{remark}
		The converse of Proposition \ref{Proposition "The general result"} is not true. An easy counter example is the case of $ A_3 $ in $ A_4 $. That is $ G=Gal(K/\mathbb{Q})=A_4 $ and $ Gal(K'/K)\cong A_3 $ where the values of $ a_K(p) $ are either $ 0,1 $ or $ 4 $.
	\end{remark}
	
	\begin{remark}
		If $ \textbf{1}_G $ denotes the trivial representation of $ G $ then it can be easily seen that $ \langle\textbf{1}_G,\chi_1\rangle_G=1 $.
	\end{remark}
	
	\section{Analytic continuation of the moment $ L $ function}\label{Section "Analytic continuation"}
	
	From now on we shall focus exclusively on primes $ p $ which are unramified in $ K $ without explicity mentioning it everytime. Suppose we denote by $ n,n',n'', h, n_h'' $ the quantities $ |G|, |N'|, |N''|, |H|, |N_h''| $ respectively. Since $ N $ is normal in $ G $, a representation of $ H=G/N $ can be seen as a representation of $ G $ via the quotient map $ G\to G/N $. Denote by $ \rho $ the representation of $ H $ given by 
	\[
	\rho = \bigoplus_{\pi\neq \textbf{1}_H}\pi^{\dim(\pi)},
	\]
	where $ \pi $ runs over all the irreducible representations of $ H $. Consider $ \rho $ as a representation of $ G $ and let $ \chi_2 $ denote the character of $ \rho $. Then $ \chi_2 $ is either equal to $ h-1 $ or equal to $ -1 $. Recall that $ a_K(p) = \chi_1(\sigma_p) $ where $ \sigma_p $ is the Frobenius element over $ p $. Suppose $ l $ is a given natural number, then the values of the various characters described so far can be found in the following table.
	
	\begin{table}[h!]
		\centering
		\caption{ }
		\label{Table 1}
		\begin{tabular}{|c|c|c|c|} 
			\hline
			$ \sigma_p=g=n_1n_2h\in G $  & $ \chi_2(g) $ & $ \chi_1(g) $   & $ a_K^l(p) $ \\
			\hline
			$ e_G $ 							& $ h-1 $ 	& $ n''$	& $ n''^l $ 		\\ 
			$ n_1\neq e_N $						& $ h-1 $	&  	$ 0$	&	$ 0 $			\\ 
			$ n_2\neq e_N $ 					& $ h-1 $	& $ n'' $	&	$ n''^l $		\\
			$ n_1\neq e_N $ and $ n_2\neq e_N $	& $ h-1 $	&   $ 0 $	&	$ 0 $			\\
			$ h\neq e_H $						& $ -1 $	& $n_h'' $	&	$ n_h''^l $		\\
			\hline
		\end{tabular}
	\end{table}

	The $ L $ series $ D_l(s) $ is absolutely convergent for $ Re(s) > 1 $ and from the multiplicative nature of $ a_K(n) $ we can deduce that $ D_l(s) $ has an Euler product in that region given by 
	\begin{equation}\label{Equation "Moment Euler product"}
		D_l(s) = \prod_p \left(1+ \frac{a_K^l(p)}{p^s} + \frac{a_K^l(p^2)}{p^{2s}} + \ldots \right),
	\end{equation}
	where the product runs over all primes $ p $. In order to estimate $ \mathcal{S}_l(X) $ we show that $ D_l(s) $ has analytic continuation to the left of the line $ Re(s)=1 $. We achieve this by imposing certain restrictions on the structure of $ G $ and writing $ D_l(s) $ as a product of certain $ L $ functions whose analytic properties are known. 
	
	For the remainder of this paper, we shall make the following assumptions,
	\begin{enumerate}[(I)]
		\item	For every $ h_1,h_2\in H\setminus\{e_H\} $, $ N_{h_1}'' = N_{h_2}'' $.\label{Assumption 1}
		\item	If, by abuse of notation we denote by $ \chi_1\otimes \chi_2 $ the tensor product of the corresponding representations, then $ L(s, \chi_1\otimes\chi_2) $ is entire.\label{Assumption 2}
		\item	The integer $ l $ satisfies $ n''^{l-1}\equiv n_h''^{l-1}\mod h $.\label{Assumption 3}
	\end{enumerate}
	
	Suppose we consider the system of linear equations given by 
	
	\begin{equation}\label{Equation "Linear System"}
		\begin{pmatrix}
			1	&	h-1	\\	1	&	-1
		\end{pmatrix}
		\begin{pmatrix}
			\alpha(l)	\\	\beta(l)
		\end{pmatrix}
		=
		\begin{pmatrix}
			n''^{l-1}\\	n_h''^{l-1}
		\end{pmatrix}.
	\end{equation}
	The system is solvable and the solution is given by 
	\begin{equation}\label{Equation "Linear System solutions"}
		\begin{pmatrix}
			\alpha(l)	\\	\beta(l)
		\end{pmatrix}
		=
		\begin{pmatrix}
			(n''^{l-1} - n_h''^{l-1})/h + n_h''^{l-1}\\	(n''^{l-1} - n_h''^{l-1})/h
		\end{pmatrix}.
	\end{equation}
	In particular the solutions are positive integers because of (\ref{Assumption 3}).
	
	\begin{theorem}\label{Theorem "Analytic continuation"}
		The $ L $ series $ D_l(s) $ has a meromorphic continuation to the plane $ Re(s) >1/2 $ with a pole of order $ \alpha(l) $ at the point $ s=1 $. Furthermore, the following relation holds,
		\begin{equation}\label{Equation "Analytic continuation D_l"}
			D_l(s) = L(s,\chi_1)^{\alpha(l)} L(s,\chi_1\otimes\chi_2)^{\beta(l)} U_l(s),
		\end{equation}
		where $ U_l(s) $ is absolutely convergent for $ Re(s) > 1/2 $.
	\end{theorem}
	
	\begin{proof}
		The proof follows by computing the Euler product representation for the function 
		\[
		D_l(s) / (L(s,\chi_1)^{\alpha(l)} L(s,\chi_1\otimes\chi_2)^{\beta(l)}).
		\]
		The Euler product of the numerator is given in \eqref{Equation "Moment Euler product"} and the Euler product of the denominator can be computed from the known Euler product formula for $ L(s,\chi_1) $ and $ L(s,\chi_1\otimes\chi_2) $. On comparing the Euler factors and using Table \ref{Table 1} and \eqref{Equation "Linear System solutions"} we see that the Euler factors of the quotient above at the unramified primes $ p $ are of the form
		\[
		\left(1 + \frac{\star}{p^{2s}} + \ldots\right).
		\]
		Therefore the quotient is absolutely convergent for $ Re(s)>1/2 $. If we denote the quotient by $ U_l(s) $, then \eqref{Equation "Analytic continuation D_l"} follows. The claim about the poles of $ D_l(s) $ follows from the observation $ L(s,\chi_1) $ has a simple pole at $ s=1 $ and from (\ref{Assumption 2}) above.
	\end{proof}

	\section{Estimates on the moments}\label{Section "Estimates on Moments"}
	
	The $ L $ series $ L(s,\chi_1) $ satisfies $ L(s,\chi_1) = \zeta(s)L(s,\chi_1') $ for some character $ \chi_1' $. Therefore we can rewrite \eqref{Equation "Analytic continuation D_l"} as 
	\begin{equation}\label{Equation "Moment zeta product"}
		D_l(s) = \zeta(s)^{\alpha(l)} L(s, \chi_1')^{\alpha(l)}L(s,\chi_1\otimes\chi_2)^{\beta(l)}.
	\end{equation}
	The estimates on $ \mathcal{S}_l(X) $ eventually boil down to subconvexity estimates on the $ L $ series occurring in \eqref{Equation "Moment zeta product"}. 
	
	By the Perron's formula (see \cite{Iwaneic book} for example), for $ 1\leq T\leq X $ we have 
	
	\begin{equation}\label{Equation "Perron's Formula"}
		\mathcal{S}_l(X) = \frac{1}{2\pi i} \int_{1+\epsilon - iT}^{1+\epsilon +iT} D_l(s) \frac{X^s}{s} ds + O\left(\frac{X^{1+\epsilon}}{T}\right).
	\end{equation}
	
	Suppose $ \theta_1, \theta_2,\theta_3, \theta_4 $ are exponents towards the $ t $ aspect subconvexity problem for\\ $ \zeta(s), L(s,\chi_1'), L(s,\chi) $ and $ L(s,\chi_1\otimes\chi_2) $ respectively. More precisely, for $ 1/2\leq \sigma \leq 1$ and $ |t| \geq 1 $, suppose that
	
	\begin{equation}\label{Equation "Subconvexity"}
		\begin{aligned}
			|\zeta(\sigma + it)| \ll |t|^{\theta_1(1-\sigma)}, \\
			|L(\sigma + it,\chi_1')| \ll |t|^{\theta_2(1-\sigma)}, \\
			|L(\sigma + it,\chi_1)| \ll |t|^{\theta_3(1-\sigma)}, \\
			|L(\sigma + it, \chi_1\otimes\chi_2)| \ll |t|^{\theta_4(1-\sigma)}.
		\end{aligned}
	\end{equation}
	
	Let $ X\gg 0 $ and let $ 1\leq T\leq X $ be a parameter to be optimally chosen later. Let
	\begin{equation}\label{Equation "I definition"}
		I =  \frac{1}{2\pi i} \int_{1+\epsilon - iT}^{1+\epsilon +iT} D_l(s) \frac{X^s}{s} ds.
	\end{equation}
	Starting from \eqref{Equation "Perron's Formula"} we shift the line of integration to the line $ Re(s) = 1/2 + \epsilon $. More precisely we integrate over the rectangle with vertices $ \{1+\epsilon -iT, 1+\epsilon +iT, 1/2 + \epsilon +iT, 1/2 + \epsilon -iT\} $ oriented in the anti clockwise direction. Suppose we let
	\begin{align}
		J_1 = \int_{1/2+\epsilon + iT}^{1/2+\epsilon - iT} D_l(s) \frac{X^s}{s} ds, \label{Equation "Vertical integral"}\\
		J_2 = \int_{1+\epsilon + iT}^{1/2+\epsilon+ iT} D_l(s) \frac{X^s}{s} ds, \label{Equation "Horizontal integral 1"}\\
		J_2 = \int_{1/2+\epsilon - iT}^{1+\epsilon- iT} D_l(s) \frac{X^s}{s} ds. \label{Equation "Horizontal integral 2"}
	\end{align}
	
	From Theorem \ref{Theorem "Analytic continuation"} and the Cauchy's theorem we get 
	
	\begin{equation}\label{Equation "I equation 1"}
		|I| = XP_l(\log(X)) + |J_1| + |J_2| + |J_3| + O\left(\frac{X^{1+\epsilon}}{T}\right),
	\end{equation}
	where $ P_l $ is a polynomial of degree equal to $ \alpha(l)-1 $. The first term in the right hand side of \eqref{Equation "I equation 1"} is the residue of the pole of $ D_l(s)X^ss^{-1} $ at the point $ s=1 $ which constitutes the main term and the other integrals are the error terms. We shall proceed to estimate them one by one.
	
	From \eqref{Equation "Subconvexity"} it follows that 
	\begin{align*}
		|J_1| &\ll X^{1/2+\epsilon}\int_{1}^{T} t^{(\alpha(l)\theta_3+\beta(l)\theta_4)/2 - 1+\epsilon}dt\\
		&\ll X^{1/2 + \epsilon}T^{(\alpha(l)\theta_3+\beta(l)\theta_4)/2 + \epsilon}.
	\end{align*}
	
	Similarly we have estimates for the horizontal integrals as
	\begin{align*}
		|J_2| + |J_3| \ll X^{1/2+\epsilon} T^{(\alpha(l)\theta_3 + \beta(l)\theta_4)/2 - 1} + X^{1+\epsilon}T^{-1}.
	\end{align*}
	
	If we choose $ T=X^\delta $ where $ \delta = 1/(\alpha(l)\theta_3 + \beta(l)\theta_4 + 2) $ then we see that 
	
	\begin{equation}\label{Equation "I estimate general"}
		|I| = XP_l(\log(X)) + O(X^{1 - \delta + \epsilon}).
	\end{equation}
	From the work of Heath-Brown \cite{Heath-Brown}, we see that $ \theta_3 $ can be chosen as $ n''/3 $.
	We record these computations as a theorem.
	
	\begin{theorem}\label{Theorem "Moments estimate general"}
		With notation as above, there exists a positive quantity $ \delta $ depending on $ K $ such that for every $ \epsilon > 0 $
		\[
		\mathcal{S}_l(X) = XP_l(\log(X)) + O(X^{1-\delta + \epsilon}),
		\]
		where $ P_l $ is as above and the implied constant depends on $ K,\epsilon $. More precisely,
		\begin{equation}\label{Equation "Delta definition"}
			\delta = \frac{1}{\alpha(l)\theta_3 + \beta(l)\theta_4 + 2}.
		\end{equation}
		In particular $ \theta_3 $ can be chosen to be $ n''/3 $.
	\end{theorem}
	
	Establishing sharp upper bounds for the constants $ \theta_1,\theta_2,\theta_3,\theta_4 $ is an active area of research with the ultimate goal being the Lindel\"off hypothesis which predicts that $ \theta_i $'s can be arbitrarily small. Nevertheless, for the purpose of Theorem \ref{Theorem "Moments estimate general"} it suffices that $ \theta_3,\theta_4 $ are positive, which is guaranteed by the convexity bound on the respective $ L $ functions. Below we mention two special cases when the error estimate in Theorem \ref{Theorem "Moments estimate general"} can be improved.
	
	Suppose that the Dedekind conjecture is true for $ K $, that is suppose that $ L(s,\chi_1') $ is entire. Observe that we can replace $ \theta_3 $ with $ \theta_1 + \theta_2 $ in \eqref{Equation "Delta definition"} and Theorem \ref{Theorem "Moments estimate general"} would still be true. Now suppose that $ \chi_1' $ is a direct sum of monomial representations. Then $ L(s,\chi_1) $ can be written as a product of Hecke $ L $ functions. Now, a recent result of Bourgain \cite{Bourgain} states that $ \theta_1 $ can be taken to be equal to $ 13/42 $ and using the subconvexity results of S\"ohne \cite{Sohne}, we can take $ \theta_2 $ as $ n''/3 - 1 $. Thus in this case we can see that $ \theta_1 + \theta_2 < \theta_3 $  and thus this would lead to a larger value for $ \delta $ and in turn a sharper error estimate in Theorem \ref{Theorem "Moments estimate general"}.
	
	If we assume the strong Artin conjecture for all $ 2 $ dimensional representations, then the recent Weyl subconvexity result of Aggarwal \cite{Aggarwal} could be used to derive similar bounds for $ \theta_1 + \theta_2 $, of course under suitable assumptions on $ \chi_1' $.
	
	\section{Some examples}\label{Section "Some Examples"}
	
	In this section we consider some examples for $ G $ to demonstrate that the restrictions imposed on $ G $ are not too heavy. We shall denote the cyclic group of order $ n $ as $ C_n $. Many examples are obtained when $ H=C_n $ so that the condition (\ref{Assumption 1}) is automatically satisfied. In particular, the case where $ H=C_2 $ gives rise to many standard groups.
	
	\subsection{The Galois case}\label{Subsection "The Galois Case"}
	As a `zeroth' example, consider the case when $ H $ is the trivial group. Although we have assumed that $ H $ is non trivial throughout the paper, the arguments throughout follow vacuously for the trivial case as well. In this situation we choose $ G=G'=N $ so that $ K $ itself is a Galois extension of $ \mathbb{Q} $. In this case we see that $ \beta(l)=0 $ so that $ \alpha(l) = n''^{l-1} $ so that $ D_l(s) = L(s,\chi_1)^{\alpha(l)}U_l(s) $ (c.f. Theorem \ref{Theorem "Analytic continuation"}). From Theorem \ref{Theorem "Moments estimate general"} we obtain the asymptotic formula of Theorem \ref{Theorem "Chand Good"}.
	
	\subsection{$ N\rtimes C_2 $ for an Abelian group $ N $}
	
	Suppose $ G=N\rtimes C_2 $ where $ C_2 $ acts via inversion. Then every subgroup $ N' $ is normal in $ N $ and is fixed by $ C_2 $. Condition (\ref{Assumption 1}) is clearly true. Since $ \chi_2 $ is a one dimensional representation of $ G $, the condition (\ref{Assumption 2}) also follows from the known theory of twisting via characters. In particular the elements fixed by the non identity element of $ C_2 $ are precisely the elements of order two. From the structure theorem of finite Abelian groups we see that $ n'' \equiv n_{-1}'' \mod 2 $, where $ -1 $ is the non identity element of $ C_2 $. Therefore clearly condition (\ref{Assumption 3}) also follows.
	
	In particular we can choose $ N $ as the cyclic group $ C_n $ and get $ N\rtimes C_2 $ as $ D_{2n} $, the dihedral group of order $ 2n $.
	
	\subsection{Symmetric Groups $ S_n $ for $ n>5 $} The symmetric group on the objects $ \{1,2,\ldots, n\} $ has the well known semi direct product representation as $ A_n\rtimes C_2 $, where $ A_n $ is the alternating group on $ n $ objects and here $ C_2 = \{e, (1,2)\} $ acts via conjugation. First suppose $ n> 5 $. Since $ A_n $ is simple in this case, the only choice of $ N' $ is the identity subgroup because we want $ N' $ to be normal in $ N $. The fixed points of $ A_n $ under conjugation by $ (1,2) $ are precisely those permutations which fix both $ \{1,2\} $. Therefore the number of the fixed elements is $ (n-1)!/2 $. Again conditions (\ref{Assumption 1}) and (\ref{Assumption 2}) are staisfied as above. Finally condition (\ref{Assumption 3}) is also true because we have assumed that $ n>5 $.
	
	\subsection{The alternating group $ A_4 $} The group $ A_4 $ has a semi direct product representation given by $ V_4\rtimes C_3 $ where $ V_4 $ is the Klien's four group. Here $ C_3 $ is identified with $ \{1, (1,2,3), (1,3,2)\} $ and $ V_4 $ is identified with $ \{1, (1,2)(3,4), (1,4)(2,3),(1,3)(2,4)\} $ and once again the action is via conjugation. Clearly only the identity in $ V_4 $ is fixed by non identity elements of $ C_3 $. Condition (\ref{Assumption 2}) follows as usual. If $ N'=\{e_N\} $ then (\ref{Assumption 3}) is true for any $ l $, if $ |N'|=2 $, then condition (\ref{Assumption 3}) holds only for odd values of $ l $.
	
	\begin{remark}
		It is worth mentioning that the inverse Galois problem for the groups mentioned above are well known and the results of this paper are truly of interest.
	\end{remark}
	
	\section*{Acknowledgments}
	
	The second named author would like to thank Teja Srinivas of the Indian Institute of Science, Bangalore and the ANTs group, Harish Chandra Research Institute for wonderful discussions. We thank the referee for valuable comments on the first draft of the paper, particularly for suggesting the authors study the dihedral case.

\end{document}